\documentclass[12pt, letterpaper]{amsart}
\usepackage[utf8]{inputenc}
\usepackage[margin=2.9cm]{geometry}
\usepackage{amsfonts}
\usepackage{amsmath}
\usepackage{graphicx}
\usepackage{amsthm}
\usepackage{amsmath,amssymb,tikz-cd}
\usepackage{hyperref}

\def\Q{\mathbb Q}
\def\mb{\mathbb}
\def\mc{\mathcal}

\def\wt{\widetilde}
\def\bs{\backslash}
\def\xr{\xrightarrow}

\def\L{\Lambda}

\newtheorem{theorem}{Theorem}[subsection]
\newtheorem{cor}{Corollary}[theorem]
\newtheorem{lemma}[theorem]{Lemma}
\newtheorem{prop}[theorem]{Proposition}
\theoremstyle{definition}
\newtheorem{definition}[theorem]{Definition}
\theoremstyle{remark}

\begin{document}
\title{The norm map on the Bloch group for quadratic extensions}
\author{V. Bolbachan}
\address{National Research University Higher School of Economics \\ Department of Mathematics \\
Usacheva str., 6, Moscow 119048 Russia.\\}
\email{vbolbachan@gmail.com}
\date{September 2019} 
\begin{abstract}
    D. Rudenko proved the homotopy invariance of the truncated polylogarithmic complexes. It follows that on these complexes there is the norm map with good proprieties. We apply his result and get the explicit formula for the norm map in the case of quadratic extensions.
\end{abstract}
\maketitle
\section{Introduction}
\subsection{A Survey}
Let $K$ be a field. We are working modulo torsion. So all abelian groups are supposed to be tensored by $\Q$. We recall the following definition.

\begin{definition}[the pre-Bloch group] 
\label{def:Bloch_group}
Denote by $\mb Q[K\bs \{0,1\}]$ the free abelian group generated by the set $K\bs\{0,1\}$. We  denote these  generators  as  $[x]_2$,  where $x\in K\bs\{0,1\}$. We define {\it the pre-Bloch group} of the field $K$ as the quotient of the group $\Q[K\bs \{0,1\}]$ by the following elements(so-called Abel five-term relations)
\begin{equation*}
\label{formula:five-relation}
[x]_2-[y]_2+[y/x]_2+\left[(1-x)/(1-y)\right]_2-\left[\dfrac {1-x^{-1}}{1-y^{-1}}\right]_2,
    \end{equation*}
where $x,y\in K\backslash\{0,1\}, x\ne y$.
\end{definition}

We recall that the Milnor $K$-theory $K_n^M(K)$ of the field $K$ is defined as the quotient of the group $\Lambda^n K^{\times}$ by the subgroup generated by the elements of the form $x\wedge(1-x)\wedge y_{n-2}, y_{n-2}\in \Lambda^{n-2}K^{\times}$. According to result of \cite{suslin1979reciprocity} for any field extension $F/K$ there is the norm map $N_{F/K}\colon K_n^M(F)\to K_n(K)$ with good proprieties.

D. Rudenko was defined the following complex placed in degrees $1,2$
$$B_2(K,n)\colon[B_2(K)\otimes
\Lambda^{n-2}K^{\times}]_a\xrightarrow{\delta_n} \Lambda^n K^{\times}.$$
Here $[B_2(K)\otimes \Lambda^{n-2}K^{\times}]_a$ is the quotient of the vector space $B_2(K)\otimes_{\Q} \Lambda^{n-2}_{\Q}K^{\times}$ by the subgroup generated by the elements of the form $[x]\otimes x\wedge y_{n-3}, y_{n-3}\in \Lambda^{n-3}K^{\times}.$ We have $H^2(B_2(K,n))\simeq K_n^M(K)$.

For any monic irreducible polynomial $P$ there is the following map of complexes

\begin{equation*}
\begin{tikzcd}[row sep=huge]
B_2(K(t))\otimes \Lambda^{n-1} K(t)^\times\arrow[r,"\delta_{n+1}"]\arrow[d,"\partial_{P,n}{[1]}"] & \Lambda^{n+1} K(t)^\times\arrow[d,"\partial_{P,n}{[2]}"]\\
B_2(K_P)\otimes \Lambda^{n-2} K_P^\times\arrow[r,"\delta_{n}"] & \Lambda^{n} K_P^\times,
\end{tikzcd}
\end{equation*}

called {\it the residue homomorphism}(see Sections 2.1 and 3.1 of \cite{rudenko2015strong}). Denote by $\partial$ the direct sum $\bigoplus\limits_{P}\partial_P\colon B_2(K(t),n+1)\to\bigoplus\limits_{P}B_2(K_P,n)$.

Following the approach of \cite{bass1973milnor}, D.Rudenko constructed the norm map $$N_{F/K}\colon H^1(B_2(F,n))\to H^1(B_2(K,n)).$$ Let us recall his construction. Main result of \cite{rudenko2015strong} is the following exact sequence:
\begin{equation}
\label{formula:Rudenko_sequence}
    0\to B_2(K,n+1)\to B_2(K(t),n+1)\xrightarrow{\partial} \bigoplus_{P}B_2(K_P,n)\to 0.
\end{equation}

Here the first map is the natural inclusion and the sum is taken over all monic irreducible polynomial over $K$. 

Let $x\in B_2(K_P,n)$ and denote by $x'\in B_2(K(t),n+1)$ some preimage of $x$ under the map $\partial$. Then the element $-\partial_{\infty}(x')$ is well-defined and we set $N_{F/K}(x):=-\partial_{\infty}(x')$. Here $\partial_{\infty}$ is the residue homomorphism at $\infty$.

There is the degree filtration $\mc F_d$ on the complex $B_2(K,n)$.  To prove the exactness of the sequence (\ref{formula:Rudenko_sequence}) D. Rudenko constructs some resolution $\wt B_2'(K,n)$ of the complex $B_2(K,n)$ and for any monic irreducible polynomial $P$ construct so-called {\it the co-residue map}
$$c_{P,n}'\colon \wt B_2'(K_P,n)\to gr_d^{\mc F} B_2(K(t),n+1),$$

such that for $d\ge 1$ and $i=1,2$ the following map

$$\bigoplus\limits_{\deg P=d}H^i(B_2(K_P,n))\simeq\bigoplus\limits_{\deg P=d}H^i(\wt B_2'(K_P,n))\to gr_d^{\mc F}H^i(B_2(K(t), n+1))$$

is inverse to the residue homomorphism $\partial$. 

\subsection{Statement of the Main Result}Let $P$ be irreducible polynomial of degree $2$ over $K$ and $F:=K[t]/P(t)$.
In the next section we will define some resolution $\wt B_2(F,n)$ of the complex $B_2(F,n)$ and modulo some simple homotopy(which is given by a simple explicit formula in the case $n=2$) construct the map $c_{P,n}\colon \wt B_2(F,n)\to B_2(K(t),n+1)/B_2(K,n+1)$. Here is our main result:

\begin{theorem}
\label{thm:main}
The map $c_{P,n}$ is well defined. If $Q$ is a monic irreducible polynomial over $K$ then the composition $\partial_Qc_{P,n}$ is equal to zero if $Q\ne P$ and is identical if $Q=P$.
\end{theorem}

\begin{cor}
\label{cor:main}
The norm map $N_{F/K}$ is given by the formula $-\partial_{\infty}c_{P,n}$. 
\end{cor}

Let us formulate these results explicitly in the case $n=2$. Consider the following map of complexes:
\begin{equation*}
\begin{tikzcd}[row sep=huge]
\Q[F\bs\{0,1\}]\oplus R_1(F^{\times})\otimes \Q[F^{\times}]\arrow[r,"\wt \delta_2"]\arrow[d,"{\pi_2[1]}"] & \Lambda^2 \Q[F^{\times}]\arrow[d,"\pi_2{[2]}"]\\
B_2(F) \arrow[r,"\delta_2"]& \Lambda^2 F^{\times}_\Q
\end{tikzcd}
\end{equation*}

Here the vector space $R_1(F^{\times})$ is freely generated by the symbols $[a,b], a,b\in F^{\times}$, the map $\wt \delta_2$ is defined by the formula $\wt \delta_2([x]_2)=[x]\wedge [1-x], \wt\delta_2([a,b]\otimes y_1)=[ab]\wedge y_1-[a]\wedge y_1-[b]\wedge y_1$, the map  $\pi_2$ is defined by the formulas $\pi_2[1]([x])=[x]_2, \pi_2[1]([a,b]\otimes y_1)=0, \pi_2[2]([x]\wedge [y])=x\wedge y$. Denote by $R_2(K)$ the subgroup $\ker \wt \delta_2\cap\ker\pi_2[1]$. The following complex is a resolution of the complex $B_2(K)$:
$$R_{2}(K)\xr{j_2} \Q[F\bs\{0,1\}]\oplus R_1(F^{\times})\otimes \Q[F^{\times}]\xrightarrow{\wt \delta_2} \Lambda^2 \Q[F^{\times}].$$
We will call it $\wt B_2(K)$. There is the natural projection $\pi_2\colon {\wt B}_2(K)\to B_2(K)$. The map $\pi_2$ is a quasi-isomorphism.

For an element $x\in F$ denote by $x_0, x_1$ the unique elements of $K$ such that $x=x_0+x_1\xi$ and set $L(x)=x_0+x_1 t\in K(t)^{\times}$. Define the element $l_x$ as follows: it is equal to $1$ if $x\in K^{\times}$ and $L(x)/x_1$ otherwise. For two elements $a,b\in F$ denote by $c$ they product. We need the following definition

\begin{definition}
Let $a,b\in F$. Define the element $X(a,b)\in B_2(K(t))$ as follows:
\begin{enumerate}
    \item If $a\in K$ or $b\in K$ then $X(a,b)=0$.
    \item If $a,b\not \in K$ but $c\in K$ then $X(a,b)=[L(c)/(L(a)L(b))]_2$.
    \item In the remaining case
    $$X(a,b)=[L(c)/(L(a)L(b))]_2-[l_{c}/l_{a}]_2-[l_{c}/l_{b}]_2.$$
\end{enumerate}
\end{definition}

%For any irreducible polynomial $Q$ over $K$ there is the residue homomorphism $\partial_Q\colon B_2(K(t))\to B_2(K_Q)$
\begin{definition}
Define the following morphism of complexes:
\begin{equation*}
\begin{tikzcd}[row sep=huge]
R_2(F)\arrow[r]\arrow[d]&  \Q[F\bs\{0,1\}]\oplus  R_1(F^{\times})\otimes \Q[F^{\times}]\arrow[r,"\wt\delta_2"] \arrow[d,swap,"n_{F/K,2}{[1]}"]& \Lambda^2 \Q[F^{\times}] \arrow[d,"n_{F/K,2}{[2]}"]\\
0\arrow[r]&
B_2(K)\arrow[r,"\delta_{2}"] & \Lambda^{2} K^{\times}
\end{tikzcd}
\end{equation*}

\begin{enumerate}
    \item We set $n_{F/K,2}[1]([x]_2)=2[x]_2$ if $x\in K$ and $0$ otherwise.
    \item If $c\in K$ the value of $n_{F/K,2}[1]([a,b]\otimes [c])$ is equal to $0$. In other case it is equal to $\partial_{l_c}(X(a,b))$.
    \item The value $n_{F/K,2}[2]([a]\wedge [b])$ is defined as follows:
    \begin{enumerate}
        \item If $a,b\in K$ then  $n_{F/K,2}[2]([a]\wedge [b])=2(a\wedge b)$.
        \item If, say, $a\in K,b\not\in K$ then it is equal to $a\wedge N(b)$ where $N(b)$ is the norm of the element $b$.
        \item If $a,b\not\in K$, but $a/b=t\in K$ then $n_{F/K,2}[2]([a]\wedge [b])=t\wedge N(b)$.
        \item In the remaining case
        $$n_{F/K,2}[2]([a]\wedge [b])=2(a_1\wedge b_1)-\dfrac {N(a)}{a_1^2}\wedge \left(b_0-\dfrac {b_1a_0}{a_1}\right)-\left(a_0-\dfrac {a_1b_0}{b_1}\right)\wedge \dfrac {N(b)}{b_1^2}.$$

    \end{enumerate}
\end{enumerate}
\end{definition}

We have the following

\begin{theorem}
\label{thm:norm_2}
Let $F/K$ be a quadratic extension. The map $n_{F/K,2}$ is actually a morpism of complexes. For any $j=1,2$ the following diagramm is commutative
\begin{equation*}
\begin{tikzcd}[row sep=huge]
H^j(\wt B_2(F))\arrow[r,"\pi_{2}"]\arrow[d,swap,"n_{F/K,2}"]& H^j(B_2(F))\arrow[ld,"N_{F/K}"]\\
H^j(B_2(K))
\end{tikzcd}
\end{equation*}
\end{theorem}

\subsection{Acknowledgments} The author is grateful to his supervisor Andrey Levin for posing the problem and useful remarks.

\section{The construction of the map \texorpdfstring{$c_{P,n}$}{TEXT}} Similarly to the case $n=2$ one can define the complex $\wt B_2(F,n)$. It has the following form:
$$\wt B_2(F,n)\colon R_n(F)\xr{j_n} R_1(F^{\times})\otimes \Lambda^{n-1}\Q[F^\times]\oplus \Q[F^{\times}]_2\otimes \Lambda^{n-2}\Q[F^{\times}]\xr{\wt \delta_n} \Lambda^{n} \Q[F^\times].$$
As in the case $n=2$, there is the natural map of complexes $\pi_n\colon \wt B_2(F,n)\to B_2(F,n)$. It is easy to see that this map is a quasi-isomorphism.

Our immediate goal is to construct a map of complexes $$c_{P,n}\colon \wt B_2(F,n)\to B_2(K(t), n+1)$$ such that the following diagram will be commutative:
\begin{equation}
\label{formula:diagramm_cPn}
\begin{tikzcd}[row sep=huge]
R_1(F^{\times})\otimes \Lambda^{n-1}\Q[F^\times]\oplus \Q[F^{\times}]_2\otimes \Lambda^{n-2}\Q[F^{\times}] \arrow[r,"\widetilde\delta_{n}"] \arrow[d,swap,"c_{P,n}{[1]}"]& \Lambda^{n} \Q[F^\times] \arrow[d,"c_{P,n}{[2]}"]\\
B_2(K(t))\otimes_a \Lambda^{n-1} K(t)^\times\arrow[r,"\delta_{n+1}"] & \Lambda^{n+1} K(t)^\times.
\end{tikzcd}
\end{equation}

%The two maps $\widetilde\delta_n,\delta_{n+1}$ are defined by the following formula
%\begin{equation}
 %   \begin{split}
  %      \widetilde\delta_{n+1}([a_1,a_2]\otimes w)&=[a_1a_2]\wedge w-[a_1]\wedge w-[a_2]\wedge w, w\in \Lambda^{n-1}\Q[F^{\times}]\\
   %     \wt \delta_{n+1}([x]_2\otimes w)&=[x]\wedge [1-x]\wedge w, w\in \Lambda^{n-2}\Q[F^{\times}]\\
    %\delta_{n+2}([x]_2\otimes w)&=x\wedge (1-x)\wedge w, w\in \Lambda^{n-1}K(t)^{\times}.
    %\end{split}
%\end{equation}}

Define the map $L\colon \Lambda^n \mb Q[F^{\times}]\to \Lambda^n K(t)^{\times}$ by the formula $L({d_1\wedge\dots\wedge d_n})=L(d_1)\wedge\dots\wedge L(d_n)$. Let $x\in K$. Define the following map:

\begin{equation*}
\begin{tikzcd}[row sep=huge]
B_2(K)\otimes \Lambda^{n-2}K^{\times} \arrow[r,"\delta_n"] \arrow[d,swap,"c_{x,n}{[1]}"]& \Lambda^{n} K^{\times} \arrow[d,"c_{x,n}{[2]}"]\\
B_2(K(t))\otimes \Lambda^{n-1} K(t)^\times\arrow[r,"\delta_{n+1}"] & \Lambda^{n+1} K(t)^\times
\end{tikzcd}
\end{equation*}

by the formulas $c_{x,n}[1]([x]\otimes d)=[x]_2\otimes (t-x)\wedge d, c_{x,n}[2](d)=(t-x)\wedge d$. We set $C_{1,n}=\bigoplus\limits_{x\in K}c_{x,n}.$ It is easy to see that $C_{1,n}$ is a morphism of complexes.

If $\deg P=1$ we can identify $K_P$ with $K$. Let us denote by $\partial_{1,n}$ the sum $\Sigma_{\deg P=1}\partial_{P,n}$. This map gives the following commutative diagram:

\begin{equation*}
\begin{tikzcd}[row sep=huge]
B_2(K(t))\otimes \Lambda^{n-1} K(t)^\times\arrow[r,"\delta_{n+1}"]\arrow[d,"\partial_{1,n+1}{[1]}"] & \Lambda^{n+1} K(t)^\times\arrow[d,"\partial_{1,n+1}{[2]}"]\\
B_2(K_P)\otimes \Lambda^{n-2} K^\times\arrow[r,"\delta_{n}"] & \Lambda^{n} K^\times.
\end{tikzcd}
\end{equation*}

We have $\partial_{1,n+1}\circ C_{1,n}=id$. Denote the map $id-C_{1,n}\circ \partial_{1,n+1}$ by $p_{n+1}$. Obviously $p_{n+1}$ is a projection to the kernel of the map $\partial_{1,n+1}$.

Now we can define the map $c_{P,n}[2]$ by the formula $$c_{P,n}[2](y_n)=p_{n+1}[2](P\wedge L(y_n)).$$
The second component of the map $c_{P,n}[1]$ is defined by the formula $$c_{P,n}[1]([x]\otimes y_{n-2})=p_{n+1}[1]([L(x)]_2\otimes P\wedge L(y_{n-2})).$$

Denote by $K(t)^{\times}_d$ the abelian group generated by the irreducible polynomial of degree not higher than $d$. Denote by $\Lambda^{'n+2}K(t)_1$ the subgroup of $\Lambda^{n+2}K(t)_1$ generated by the elements of the form $(t-x)\wedge w_{n+1}, x\in K, w_{n+1}\in \Lambda^{n+1}K(t)_1$ and let $B_2(K(t))_1'$ be the subgroup of $B_2(K(t))$ generated by the elements of the form $[(t-a)/(t-b)]_2, a,b\in K$.
It is easy to see that the following sequence is exact
$$B_2(K(t))_1'\otimes \Lambda^{n-1}K(t)^{\times}_1\to \Lambda^{'n+1} K(t)^{\times}_1\xr{\partial_{1,n+1}[2]} \bigoplus\limits_{x\in\mb A^1(K)}\Lambda^{n}K^{\times}\to 0.$$

Since we are work over $\mb Q$ there is a homotopy $h_{n+1}\colon  \Lambda^{'n+1} K(t)^{\times}_1\to B_2(K(t))_1'\otimes \Lambda^{n-1}K(t)^{\times}_1$ such that $\delta_{n+1} h_{n+1}=p_{n+1}$ and $h_{n+1}((t-a)\wedge y_{n-1})=0, y_{n-1}\in \L^n K^{\times}$. In the case $n=2$ we can take $h_2(t-a\wedge y_1)=0, h_2((t-a)\wedge (t-b))=-[(t-b)/(a-b)]_2, a,b,y_1 \in K$.

We recall that we defined the element $X(a,b)$ in the previous section. Denote the element $l_{c}\wedge N_{c}-l_{a}\wedge N_{a}-l_{b}\wedge N_{b}$ by $Z(a,b)$. Obviously, the element $Z(a,b)\wedge L(d)$ lies in the subgroup $\Lambda^{'n+1} K(t)^{\times}_1$

Define the map $c_{P,n}[1]$ on the component $R_1(F^{\times})\otimes \Lambda^{n-1}\Q[F^*]$ by the formula $$c_{P,n}[1]([a,b]\otimes y_{n-1})=p_{n+1}[1](-X(a,b)\otimes L(y_{n-1}))+h_{n+1}(p_{n+1}(Z(a,b)\wedge L(y_{n-1}))).$$

\section{The proofs}
\subsection{The diagram (\ref{formula:diagramm_cPn}) is commutative}
\begin{prop}
\label{prop:second_component}
On the component $\Q[F]_2\otimes \Lambda^{n-2}\Q[F^{\times}]$ we have an equality $$c_{P,n}[2]\wt\delta_n=\delta_{n+1}c_{P,n}[1].$$
\end{prop}
\begin{proof}
We have
\begin{equation*}
    \begin{split}
        &\delta_{n+1}c_{P,n}[1]([x]\otimes y_{n-2})=\delta_{n+1}p_{n+1}[1]([L(x)]_2\otimes P\wedge L(y_{n-2}))=\\
        &=p_{n+1}[2]\delta_{n+1}(([L(x)]_2\otimes P\wedge L(y_{n-2}))=p_{n+1}[2](L(x)\wedge L(1-x)\wedge P\wedge L(y_{n-2}))=\\
        &=c_{P,n}[2]\wt\delta_n([x]_2\wedge y_{n-2}).
    \end{split}
\end{equation*}
\end{proof}

Define the elements $N_a\in K^{\times}$ as follows. If $a\in K$, then $N_a=1$. In the remaining case, $N_a=N(a)/a_1^2$ where $N(a)$ is norm of the element $a$. We will denote by $c$ the product $ab$.

\begin{lemma}
\label{lemma:differential_of_X}
We have 

$$\delta(X(a,b))=-\left(P\wedge \dfrac {L(c)}{L(a)L(b)}\right)+Z(a,b).$$
\end{lemma}

We will prove this lemma in the end of this section.
\begin{prop}
\label{prop:formula}
The following equality holds
\begin{equation*}
    \begin{split}
        &\delta_{n+1}(p_{n+1}[1](-X(a,b)\wedge L(y_{n-1})))=\\
&=c_{P,n}[2]\wt \delta_{n}([a,b]\otimes y_{n-1})-p_{n+1}[2](Z(a,b)\wedge L(y_{n-1})).
    \end{split}
\end{equation*}
\end{prop}

\begin{proof}[Proof of Proposition \ref{prop:formula}]
By lemma \ref{lemma:differential_of_X}, we have
\begin{equation*}
\begin{split}
&\delta_{n+1}(p_{n+1}[1](-X(a,b)\wedge L(y_{n-1})))=p_{n+1}[2]\delta_{n+1} (-X(a,b)\wedge L(y_{n-1}))=\\
&=p_{n+1}[2]\left(P\wedge  \dfrac {L(c)}{L(a)L(b)}\wedge L(y_{n-1})\right)-p_{n+1}[2](Z(a,b)\wedge L(y_{n-1})).
\end{split}
\end{equation*}
The first term is equal to $c_{P,n}[2]\wt\delta_n([a,b]\otimes y_{n-2})$. The proposition is proven.
\end{proof}

\begin{prop}
\label{prop:diagramm_is_commutative}
The diagram (\ref{formula:diagramm_cPn}) is commutative.
\end{prop}

\begin{proof}
By Proposition \ref{prop:formula}, we have:
\begin{equation*}
\begin{split}
\delta_{n+1}c_{P,n}[1]([a,b]\otimes y_{n-1})&=\delta_{n+1}p_{n+1}[1](-X(a,b)\wedge L(y_{n-1}))+\\
+\delta_{n+1}h_{n+1}p_{n+1}[2](Z(a,b)\wedge L(y_{n-1}))&=c_{P,n}[2]\wt \delta_{n}([a,b]\otimes y_{n-1})-\\
-p_{n+1}[2](Z(a,b)\wedge L(y_{n-1}))&+\delta_{n+1}h_{n+1}p_{n+1}[2](Z(a,b)\wedge L(y_{n-1})).
\end{split}
\end{equation*}
So it is enough to prove the following formula:
$$p_{n+1}[2](Z(a,b)\wedge L(y_{n-1}))=\delta_{n+1}h_{n+1}p_{n+1}[2](Z(a,b)\wedge L(y_{n-1})).$$
By the definition of the homotopy $h_{n+1}$ we have: $\delta_{n+1}h_{n+1}p_{n+1}=p_{n+1}^2=p_{n+1}$.

The commutativity of \ref{formula:diagramm_cPn} on the first component is proven. Its commutativity on the second component follows from the Proposition \ref{prop:second_component}. 
\end{proof}

\subsection{Proof of Theorem \ref{thm:main} and Theorem \ref{thm:norm_2}} \begin{proof}[The proof of Theorem \ref{thm:main}]
Let us prove that the following map is a well-defined morphism of complexes:
\begin{equation}
\begin{tikzcd}[row sep=huge]
R_n(F)\arrow[r,"j_n"]\arrow[d,"0"]&R_1(F^{\times})\otimes \Lambda^{n-1}\Q[F^\times]\oplus \Q[F^{\times}]_2\otimes \Lambda^{n-2}\Q[F^{\times}] \arrow[r,"\widetilde\delta_{n}"] \arrow[d,swap,"c_{P,n}{[1]}"]& \Lambda^{n} \Q[F^\times] \arrow[d,"c_{P,n}{[2]}"]\\
0\arrow[r]& B_2(K(t))\otimes_a \Lambda^{n-1} K(t)^{\times}/B_2(K)\otimes_a \Lambda^{n-1} K^{\times}\arrow[r,"\delta_{n+1}"] & \Lambda^{n+1} K(t)^{\times}/\Lambda^{n+1} K^{\times}.
\end{tikzcd}
\end{equation}
By Proposition \ref{prop:diagramm_is_commutative} we only need to show that the composition $c_{P,n}[1]j_n$ is zero. By construction of the map $c_{P,n}$, for any monic irreducible polynomial $Q$ we have $\partial_{Q,n}c_{P,n}[1]=0$ if $Q\ne P$ and $\partial_{Q,n}c_{P,n}[1]=\pi_n[1]$ otherwise. It follows that $\partial_Qc_{P,n}[1]j_n=0$ for any $Q$. So by the main result of \cite{rudenko2015strong} the composition $c_{P,n}[1]j_n$ lies in $B_2(K)\otimes_a\L^{n-1}K^{\times}$. So the map is well defined. The second statement has already been proven.
\end{proof}

\begin{proof}[Proof of Theorem \ref{thm:norm_2}]
It is easy to that the map $n_{F/K,2}$ is equal to the composition $-\partial_{\infty}c_{P,2}$. So the statement follows from Corollary \ref{cor:main}.
\end{proof}

\subsection{Proof of Lemma \ref{lemma:differential_of_X}}
The following lemma is obtained by a direct computation.
\begin{lemma}
\label{Lemma_about_norm}
Let $a,b\in F$. The following statements are true(in the last two statements we need $a_1,b_1\ne 0$):
\begin{enumerate}
\item $L(a)L(b)-L(c)=a_1b_1P$
    \item $l_{a}-l_{c}=a_1b_1/(c)_1N_{a}.$
    \item $N_{c}=\left(a_1b_1/(c)_1\right)^2N_{a}N_{b}$.
\end{enumerate}
\end{lemma}

\begin{proof}[Proof of Lemma \ref{lemma:differential_of_X}]
\begin{enumerate}
    \item If $a,b\in K$ the statement is obvious.
    \item Let us assume that $a,b\not \in K$, but $c\in K$. By the first statement of the Lemma \ref{Lemma_about_norm}
    \begin{equation*}
        \delta_2\left(\dfrac{L(c)}{L(a)L(b)}\right)=\dfrac{L(c)}{L(a)L(b)}\wedge \dfrac{Pa_1b_1}{L(a)L(b)}=\dfrac{L(c)}{L(a)L(b)}\wedge P+
        l_al_b\wedge \dfrac c{a_1b_1}
    \end{equation*}
    So it is enough to prove that $N_{a}=N_{b}=-c/(a_1b_1)$. 
    Since $c\in K$ and $b=\bar {a}\dfrac{c}{N(a)}$ we have $b_1=(\bar a)_1\dfrac{c}{N(a)}=-a_1\dfrac{ab}{N(a)}$. So $c/(a_1b_1)=\dfrac{ab}{-a_1^2ab/N(a)}=-N_{a}$.
\item
Using the first statement of the Lemma \ref{Lemma_about_norm} we have

    \begin{align*}
         &\delta_2\left(\dfrac{L(c)}{L(a)L(b)}\right)=\dfrac{L(c)}{L(a)L(b)}\wedge \dfrac{Pa_1b_1}{L(a)L(b)}
    =\left(\dfrac {l_{ab}}{l_{a}l_{b}}\cdot \dfrac{c_1}{a_1b_1}\right)\wedge\left(\dfrac P{l_{a}l_{b}}\right)=\\
        &=\dfrac{L(c)}{L(a)L(b)}\wedge P-l_{ab}\wedge l_{a}-l_{ab}\wedge l_{b}+l_{a}l_{b}\wedge \dfrac{c_1}{a_1b_1}
    \end{align*}

Using the second statement of the Lemma \ref{Lemma_about_norm} we have
    \begin{equation*}
        \delta_2\left(\dfrac{l_{ab}}{l_{a}}\right)=\dfrac{l_{ab}}{l_{a}}\wedge \dfrac{a_1b_1N_{a}}{l_{a}c_1}, \quad \delta_2\left(\dfrac{l_{ab}}{l_{b}}\right)=\dfrac{l_{ab}}{l_{b}}\wedge \dfrac{a_1b_1N_{a}}{l_{b}c_1} 
    \end{equation*}

 So we get
 
 \begin{equation*}
     \delta_2\left(X(a,b)\right)=\dfrac{L(c)}{L(a)L(b)}\wedge P+l_{ab}\wedge \left(\dfrac {a_1b_1}{(ab_1)}\right)^2N_aN_b-l_{a}\wedge N_{a}-l_{b}\wedge N_{b}.
\end{equation*}
    Now the statement follows from the last statement of the Lemma \ref{Lemma_about_norm}.
\end{enumerate}
\end{proof}

\end{document}